\documentclass[english]{article}
\usepackage{preamble}
\usepackage{quiver}
\usepackage{ amssymb }
\usepackage{ wasysym }

\crefname{thm}{Theorem}{Theorems}
\crefname{lem}{Lemma}{Lemmas}
\crefname{prop}{Proposition}{Propositions}
\crefname{cor}{Corollary}{Corollaries}
\crefname{defn}{Definition}{Definitions}
\crefname{rem}{Remark}{Remarks}
\Crefname{thm}{Theorem}{Theorems}
\Crefname{lem}{Lemma}{Lemmas}
\Crefname{prop}{Proposition}{Propositions}
\Crefname{cor}{Corollary}{Corollaries}
\Crefname{defn}{Definition}{Definitions}
\Crefname{rem}{Remark}{Remarks}

\begin{document}

\title{Cofinality via Weighted Colimits}

\begin{titlepage}
    \maketitle

    \begin{abstract}
        We prove a refinement of Quillen's Theorem A, providing necessary and sufficient conditions for a functor to be cofinal with respect to diagrams valued in a fixed $\infty$-category. 
        We deduce this from a general duality phenomenon for weighted colimits, which is of independent interest.
        As a sample application, due to Betts and Dan-Cohen, we describe a simplified formula for the free $\EE_\infty$-algebra on an $\EE_0$-algebra in a stable rational $\infty$-category .  
    \end{abstract}
    % \date{\today}.

    \begin{figure}[h]
        \centering
        \captionsetup{justification=centering}

        \includegraphics[width=0.5\linewidth]{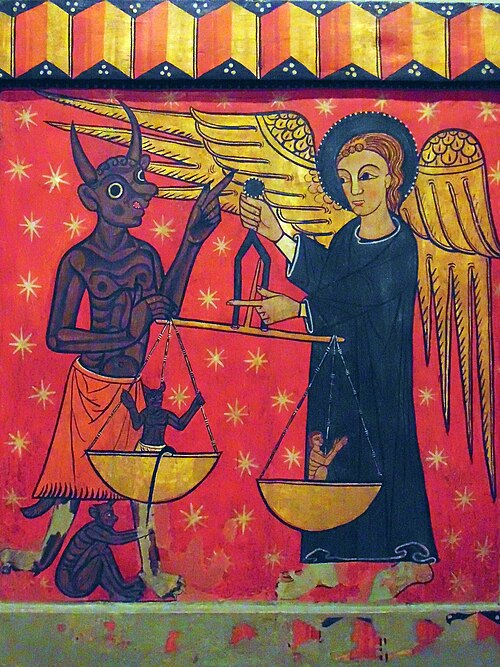}
        \caption*{\scriptsize{Judgement of the Soul (Judici particular de l'\'anima), 13th century, Catalonia. \\Photo by Manel Zaera, via Wikimedia Commons, licensed under CC BY 2.0}}
    \end{figure}
\end{titlepage}

\tableofcontents{}

%%%%%%%%%%%%%%%%%%%%%%%%%%%%%%%%%%%%%%%%%%%%%%%%%%%%%%%%%%%%%%%%%%%%%%%%%%%%%%%%%%%%%%%%%%
%%%%%%%%%%%%%%%%%%%%%%%%%%%%%%%%%%%%%%%%%%%%%%%%%%%%%%%%%%%%%%%%%%%%%%%%%%%%%%%%%%%%%%%%%%
\section{Introduction}
%%%%%%%%%%%%%%%%%%%%%%%%%%%%%%%%%%%%%%%%%%%%%%%%%%%%%%%%%%%%%%%%%%%%%%%%%%%%%%%%%%%%%%%%%%
%%%%%%%%%%%%%%%%%%%%%%%%%%%%%%%%%%%%%%%%%%%%%%%%%%%%%%%%%%%%%%%%%%%%%%%%%%%%%%%%%%%%%%%%%%

It is a fundamental fact of category theory that the colimit of a diagram depends functorially  on its \textit{shape}. Namely, given a functor $X\colon I \to \cC$ viewed as an $I$-shaped diagram in $\cC$, and a functor $f\colon J\to I$, viewed as a re-parametrization of the shape of the diagram $X$, there is a canonical comparison map on colimits,
\[
    \alpha \colon \colim_J (X\circ f) \too \colim_I X \qin \cC.
\]
It is clear that when $f$ is an equivalence of categories, $\alpha$ is an isomorphism for all $\cC$ and $X$. Interestingly, this may be the case even when $f$ is not an equivalence of categories. For example, if $J:= \pt \to I$ picks the terminal object of $I$. Such functors are called \textit{cofinal}.\footnote{Or `left cofinal', `final' or `terminal', depending on the author.} 

While the definition of cofinality for $f\colon J\to I$ quantifies over all $I$-shaped diagrams in all possible categories $\cC$, there is a classical criterion for cofinality in terms of the \textit{lax fibers} (i.e.\ comma categories) of $f$. These are defined for each $a \in I$ by the following pullback diagram of categories:
% https://q.uiver.app/#q=WzAsNCxbMSwwLCJJX3thL30iXSxbMSwxLCJJIl0sWzAsMSwiSiJdLFswLDAsIkpfe2EvfSJdLFswLDFdLFsyLDEsImYiXSxbMywyXSxbMywwXSxbMywxLCIiLDEseyJzdHlsZSI6eyJuYW1lIjoiY29ybmVyIn19XV0=&macro_url=https%3A%2F%2Fraw.githubusercontent.com%2FLior-Yanovski%2Fquiver%2Fmaster%2FLatex%2520Macros%2520preamble
\[\begin{tikzcd}
	{J_{a/}} & {I_{a/}} \\
	J & I.
	\arrow[from=1-1, to=1-2]
	\arrow[from=1-1, to=2-1]
	\arrow["\lrcorner"{anchor=center, pos=0.125}, draw=none, from=1-1, to=2-2]
	\arrow[from=1-2, to=2-2]
	\arrow["f", from=2-1, to=2-2]
\end{tikzcd}\]
Explicitly, an object of $J_{a/}$ is an object $b$ of $J$ together with a map $a\to f(b)$ in $I$, and a morphism between two such is a morphism $b_1 \to b_2$ in $J$ such that the obvious triangle in $I$ commutes: 
% https://q.uiver.app/#q=WzAsMyxbMSwwLCJhIl0sWzAsMSwiZihiXzEpIl0sWzIsMSwiZihiXzIpIl0sWzAsMV0sWzAsMl0sWzEsMl1d&macro_url=https%3A%2F%2Fraw.githubusercontent.com%2FLior-Yanovski%2Fquiver%2Fmaster%2FLatex%2520Macros%2520preamble
\[\begin{tikzcd}
	& a \\
	{f(b_1)} && {f(b_2).}
	\arrow[from=1-2, to=2-1]
	\arrow[from=1-2, to=2-3]
	\arrow[from=2-1, to=2-3]
\end{tikzcd}\]

\begin{thm}\label[thm]{Classical_Cof_Crit}
    A functor $f\colon J\to I$ is cofinal if and only if $J_{a/}$ is connected and non-empty for all $a \in I$.\footnote{We are not aware of the origin of this theorem, but an enriched version seems to have appeared first in \cite{kelly64enriched}.}
\end{thm}

In his work on algebraic $K$-theory, Quillen used a homotopical variant of the above criterion in the formulation of his celebrated `Theorem A'.

\begin{thm}[Quillen {\cite[Theorem A]{quillen2006higher}}]
    Let $f\colon J \to I$ be a functor. The induced map of simplicial sets 
    \[
        N(f)\colon N(J)\too N(I)
    \]
    is a weak homotopy equivalence if for all $a\in I$, the simplicial set $N(J_{a/})$ is weakly contractible. 
\end{thm}

It was subsequently realized that this result follows from a direct $\infty$-categorical analogue of \Cref{Classical_Cof_Crit}. The notions of diagrams, colimits, cofinality, lax fibers etc.\!\! extend naturally to $\infty$-categories. Moreover, ``weak contractibility'' for an $\infty$-category $I$ has a natural intrinsic interpretation. 
It means that the $\infty$-groupoidification $|I|$ is contractible (i.e.,\ equivalent to the trivial one point groupoid).\footnote{In the spirit of the homotopy hypothesis, we conflate $\infty$-groupoids with spaces.} With these $\infty$-categorical adaptations, we have the following:

\begin{thm}[{Joyal \cite{joyal2008theory}, Lurie \cite[Theorem 4.1.3.1]{HTT}}]\label[thm]{Homotopy_Cof_Crit}
    A functor $f\colon J \to I$ of $\infty$-categories is cofinal if and only if $|J_{a/}| \simeq \pt$ for all $a \in I$.
\end{thm}

Quillen's original Theorem A follows from the above by considering the constant $I$-shaped diagram $I \to \Spc$ on the terminal object $\pt \in \Spc$ where $\Spc$ is the $\infty$-category of spaces (equivalently, $\infty$-groupoids).  

Thus, for a functor $f\colon J \to I$ of ordinary categories, we have two different notions of cofinality: 
\begin{enumerate}
    \item ``classical cofinality'' with respect to diagrams in ordinary categories.
    \item ``homotopy cofinality'' with respect to diagrams in all $\infty$-categories.
\end{enumerate}
In particular, \Cref{Classical_Cof_Crit} and \Cref{Homotopy_Cof_Crit} are incomparable, none of which implies the other in an obvious way. Nevertheless, for $f$ to be cofinal for $\cC$-valued diagrams, intuition suggests that it suffices for the spaces $|J_{a/}|$ to ``appear contractible'' for the specific $\infty$-category $\cC$. So, for example, if $\cC$ is an ordinary category, its mapping spaces are 0-truncated and hence cannot distinguish between `contractible' and merely `connected'. 

The goal of this note is to make the above intuition precise. We begin with some definitions.

\begin{defn}
    Let $\cC$ be an $\infty$-category. A functor $f\colon J \to I$ of $\infty$-categories is called \textit{$\cC$-cofinal} if for all diagrams $X\colon I \to \cC$, the induced comparison map is an isomorphism\footnote{We adopt the convention that an isomorphism of colimits implies, in particular, that one exists if and only if the other exists, but does not assert the existence of the colimits in general.}
    \[
        \colim_J (X\circ f) \iso \colim_I X \qin \cC.
    \]
\end{defn}

Next, we want to formalize the idea that a space $A$ appears to be contractible in the eyes of a specific $\infty$-category $\cC$.

\begin{defn}
    Let $\cC$ be an $\infty$-category and $A$ a space (i.e.,\ $\infty$-groupoid). 
    \begin{enumerate}
        \item For $X \in \cC$, we denote by $A \otimes X$ the colimit in $\cC$ of the $A$-shaped constant diagram on $X$.
        \item We say that $A$ is \textit{$\cC$-acyclic} if for every $X \in \cC$ the constant colimit $A \otimes X$ exists and the map $A \to \pt$ induces an isomorphism $A \otimes X \iso X$.
    \end{enumerate}
\end{defn}

\begin{rem}
    We have chosen the ``homological'' variant of acyclicity. But we could just as well consider the ``cohomological'' variant by asking the map $A\to \pt$ to induce an isomorphism on limits $Y \iso Y^A$ for all $Y \in \cC$. The two variants are equivalent, as by the Yoneda lemma both coincide with having canonical isomorphisms
    \[
        \Map(X,Y) \iso \Map(X,Y)^A \qin \Spc,
    \]
    for all $X,Y\in \sC$.
\end{rem}

With these definitions, we can state our first main result.

\begin{theorem}[\Cref{Local_Cof_Crit}]\label[thm]{Local_Cof_Crit_Intro}
    Let $\cC$ be an $\infty$-category. A functor $f \colon J \to I$ of $\infty$-categories is $\cC$-cofinal if $|J_{a/}|$ is $\cC$-acyclic for all $a \in I$. The converse holds provided $\cC$ admits enough colimits.\footnote{We refer to \Cref{Local_Cof_Crit} for the precise condition.}
\end{theorem}

We observe that $|J_{a/}|$ is $\cC$-acyclic precisely when $J_{a/}\to \pt$ is $\cC$-cofinal. Thus, \Cref{Local_Cof_Crit_Intro} can be informally summarized as saying that $\cC$-cofinality is a ``lax fiberwise condition''. 

% \begin{rem}
%     The notion of a lax fiber is a special case of the more general notion of a \textit{lax pullback}. Given a diagram of $\infty$-categories $K\oto{g} I \ofrom{f} J$, its lax pullback is the $\infty$-category $K\vec{\times}_I J$ whose objects are roughly triples of $a \in K$, $b\in J$ and a map $g(a) \to f(b)$ in $I$. The case where $g\colon \pt \to I$ picks the object $a$ reproduces the lax fiber of $f$ at $a$. The combination of \Cref{Local_Cof_Crit_Intro} and the pasting lemma for lax pullbacks [I think this is actually false!] implies that $\cC$-cofinal functors are stable under \textit{lax base-change}. We leave the details to the reader. 
% \end{rem}

% \begin{rem}
%     The statement above is not an equivalence in general, since the $\infty$-category $\cC$ may not admit all colimits. In such a situation, the fact that $f$ is $\cC$-cofinal cannot necessarily be detected by the fibers $J_{a/}$. 
    
%     On the other hand, if $\cC$ admits all the colimits under consideration, then $\cC$-cofinality is stable under pullbacks. Therefore the map
%     \[
%         J_{a/} \longrightarrow I_{a/} \longrightarrow \{a\}
%     \]
%     is $\cC$-cofinal, which is equivalent to the $\cC$-acyclicity of $|J_{a/}|$.
% \end{rem}

For many naturally occurring families of $\infty$-categories, one can identify a substantial collection of acyclic spaces leading to enlightening specializations of \Cref{Local_Cof_Crit_Intro}.
\begin{enumerate}
    \item If $\cC$ is an $(n,1)$-category, namely, all mapping spaces are $(n-1)$-truncated, then every $n$-connected space is $\cC$-acyclic. This provides a natural interpolation between \Cref{Classical_Cof_Crit} and \cref{Homotopy_Cof_Crit}.
    After an initial version of this work was completed, we became aware that this result was proved in \cite{other_quillen_paper}.

    \item If $\cC$ is stable and $E$-local for a certain spectrum $E$, then every $E$-acyclic space (i.e.,\ having the $E$-homology of a point) is $\cC$-acyclic. In the special case of a stable rational $\infty$-category, this includes all connected $\pi$-finite spaces.

    \item If $\cC$ is higher semiadditive of height less than or equal to $n$, then every $\pi$-finite and $n$-connected space is $\cC$-acyclic \cite[Theorem 3.2.7]{AmbiHeight}. In particular, this includes stable $T(n)$- (and hence $K(n)$-)local $\infty$-categories, which are the higher chromatic analogues of the stable rational $\infty$-categories. 
\end{enumerate}

We further illustrate the utility of \Cref{Local_Cof_Crit_Intro} by describing an application due to Betts and Dan-Cohen to a concrete formula for the free $\EE_\infty$-algebra on a pointed object in a stable rational symmetric monoidal $\infty$-category.

\begin{corollary}[\cite{Ishai_Paper}, see \Cref{Ishai}]\label[cor]{Ishai_Intro}
    Let $\cC$ be a stable rational symmetric monoidal $\infty$-category and let $\one \to X$ be a pointed object in $\cC$ viewed as an $\EE_0$-algebra. The underlying object of the free $\EE_\infty$-algebra on the $\EE_0$-algebra $X$ is given by the sequential colimit
    \[
        \colim \big(
        \one \to X \to X^{\otimes 2}_{h\Sigma_2} \to X^{\otimes 3}_{h\Sigma_3} \to \dots \big) \qin \cC.
    \]
\end{corollary}

We deduce \cref{Local_Cof_Crit_Intro} from a more general quantitative result, which we believe to be of independent interest, regarding \textit{weighted} colimits. These constitute a two-variable extension of the colimit operation which in addition to a diagram $X\colon I \to \cC$, takes as input also a \textit{weight} functor $W\colon I^\op \to \Spc$. More in detail, we have the \textit{twisted arrow} $\infty$-category of $I$ with its target-source right fibration\footnote{This is the usual source-target right fibration composed with $I\times I^\op \iso I^\op\times I$. The \textit{left} fibration $\Tw(I)^\op \to I^\op\times I$ classified by $\Map$ is used to define weighted \textit{limits}.} 
\[
    \Tw(I) \oto{\ (t,s)\ } I^\op \times I,
\]
classifying the mapping space functor $\Map \colon I^\op \times I \to \Spc$. 

\begin{defn}[{\cite[Definition 2.8]{gepner2017lax}}]
    Given functors $X\colon I \to \cC$ and $W\colon I^\op \to \Spc$, the \textit{colimit of $X$ weighted by $W$}, denoted $\colim_I^W X$, is the colimit of the composite functor
    \[
        \Tw(I) \oto{\ (t,s)\ } I^\op \times I \oto{W \otimes X} \cC.
    \]
\end{defn}

Our second main result is a certain duality for weighted colimits between restricting the diagram $X\colon I \to \cC$ along $f\colon J \to I$, and left Kan extending the weight $W \colon J^\op \to \Spc$ along $f^\op\colon J^\op \to I^\op$ (denoted $f^\op_!W$).  

\begin{theorem}[\cref{Colim_Duality}]\label[thm]{Colim_Duality_Intro}
    Let $\cC$ be an $\infty$-category and $f\colon J \to I$ a functor. For every diagram $X\colon I\to \cC$ and weight $W\colon J^\op \to \Spc$, we have an isomorphism
    \[
        \colim_J^W (X\circ f) \simeq 
        \colim_I^{f^\op_! W} X \qin \cC.
    \]
\end{theorem}

\begin{rem}
    The weighted colimit $\colim_I^W X$ of a diagram $X\colon I \to \cC$ with respect to a weight $W\colon I^\op \to \Spc$ over a category $I$ can be likened to the integral $\int_M\varphi \,\mathrm{d}\mu$ of a function $\varphi \colon M \to \RR$ with respect to a measure $\mu$ on the space $M$. Given a map of spaces $f\colon N \to M$ and a measure $\nu$ on $N$, we have 
    \[
        \int_N (\varphi \circ f)\,\mathrm{d}\nu = \int_M \varphi \,\mathrm{d}(f_*\nu).
    \]
    \Cref{Colim_Duality_Intro} can be seen as the analogue of this identity where the left Kan extension $f_!W$ plays the role of the pushforward measure $f_*\nu$.
\end{rem}

The usual colimit along $J$ can be recovered by taking the weight $W$ to be the constant functor on $\pt \in \Spc$. We show that in this special case we get $(f^\op_!W)(a) \simeq |J_{a/}|$, obtaining the following:

\begin{corollary}[\cref{Cof_Quant}]\label[cor]{Cof_Quant_Intro}
    Let $\cC$ be an $\infty$-category. For every functor $f\colon J\to I$ and a diagram $X\colon I \to \cC$, we have an isomorphism
    \[
        \colim_J (X\circ f) \simeq \colim_I^{|J_{(-)/}|} X \qin \cC.
    \]
\end{corollary}

In particular, if all the spaces $|J_{a/}|$ are $\cC$-acyclic, then the right-hand side of the above isomorphism coincides with the ordinary colimit of $X$ over $I$, thus proving \Cref{Local_Cof_Crit_Intro}.

\begin{rem}
    Using the pointwise formula for left Kan extension along $f\colon J \to I$, one gets another formula for $\colim_J (X\circ f)$ as a certain colimit over $I$,
    \[
        \colim_J (X\circ f) \simeq 
        \colim_I \big(\colim_{(b,f(b)\to a)\in J_{/a}} X(f(b))\big). 
    \]
    This involves the \textit{oplax fibers} $J_{/a}$ (rather than $J_{a/}$) and, while having its own uses, doesn't seem to lead to a proof of \Cref{Local_Cof_Crit_Intro}. Thus, the introduction of the additional weight variable to the colimit operation is essential for our argument.
\end{rem}

\paragraph{Acknowledgments.} We are grateful to Ishai Dan-Cohen for proposing  \cref{Ishai_Intro} as an application of a conjectural refinement of Quillen's theorem A and for many useful suggestions on a previous draft. We thank the other participants of the Seminarak where this refinement was essentially proved, inspiring the main ideas of this note. Special thanks to Shay Ben-Moshe for useful discussions about the subject matter and to Dorin Boger for useful comments on an earlier draft. 
We also thank Niko Naumann for renewing our interest in \cref{Local_Cof_Crit_Intro}, encouraging us to write this note. 
The first author was supported by the DFG through SFB 1085 ``Higher Invariants''.

%%%%%%%%%%%%%%%%%%%%%%%%%%%%%%%%%%%%%%%%%%%%%%%%%%%%%%%%%%%%%%%%%%%%%%%%%%%%%%%%%%%%%%%%%%
%%%%%%%%%%%%%%%%%%%%%%%%%%%%%%%%%%%%%%%%%%%%%%%%%%%%%%%%%%%%%%%%%%%%%%%%%%%%%%%%%%%%%%%%%%
\section{Presentable envelope}
%%%%%%%%%%%%%%%%%%%%%%%%%%%%%%%%%%%%%%%%%%%%%%%%%%%%%%%%%%%%%%%%%%%%%%%%%%%%%%%%%%%%%%%%%%
%%%%%%%%%%%%%%%%%%%%%%%%%%%%%%%%%%%%%%%%%%%%%%%%%%%%%%%%%%%%%%%%%%%%%%%%%%%%%%%%%%%%%%%%%%

Our first task is to reduce the discussion to the case of presentable $\infty$-categories. We achieve this by observing that every $\infty$-category $\cC$ admits a fully faithful embedding $\widehat{j}\colon \cC \into \widehat{\cC}$ such that,
\begin{enumerate}
    \item $\widehat{\cC}$ is presentable.
    \item $\widehat{j}$ preserves all small\footnote{By enlarging the universe, one may regard every given diagram as ``small'' and adjust the notion of \textit{presentability} correspondingly. The problem arises only when one attempts to deal with all diagrams at once, and then some uniform bound on the sizes is required.} colimits that exist in $\cC$.
\end{enumerate}
This follows from the results of sections 5.3.6 and 5.5.4 of \cite{HTT}. We briefly explain how for the convenience of the reader. The starting point is the \textit{free co-completion} given by the Yoneda embedding
\[
    j \colon \cC \intoo \Psh(\cC) := \Fun(\cC^\op, \Spc).
\]
The $\infty$-category $\Psh(\cC)$ is presentable and in particular admits all small colimits. However, the fully faithful embedding $j$ does not send colimits in $\cC$ to colimits in $\Psh(\cC)$. This can be fixed by applying \cite[Theorem 5.3.6.2]{HTT}, which constructs a fully faithful embedding  
\[
    j_\sR\colon \cC \intoo \Psh_\sR(\cC),
\]
where $\sR$ is any collection of colimiting cones in $\cC$, and such that
\begin{enumerate}
    \item $\Psh_\sR(\cC)$ admits all small colimits.
    \item $j_\sR$ maps all the colimiting cones in $\sR$ to colimiting cones in $\Psh_\sR(\cC)$. 
\end{enumerate}
Unpacking the proof, the $\infty$-category $\Psh_\sR(\cC)$ is constructed as a full subcategory of $\Psh(\cC)$ by the following procedure: Every cone diagram in $\sR$ induces a map from the image of the cone point (the colimit of the diagram in $\cC$) to the colimit of the corresponding diagram (without the cone point) in $\Psh(\cC)$. Let $S$ be the collection of all these Yoneda assembly maps. The $\infty$-category $\Psh_\sR(\cC)$ is defined to be the full subcategory of $\Psh(\cC)$ spanned by the $S$-local objects in the sense of \cite[Definition 5.5.4.1]{HTT}. Namely, 
\[
    \Psh_\sR(\cC) := S^{-1}\Psh(\cC).
\]
In particular, the Yoneda embedding itself is the special case $\sR =\es$.

Using this explicit presentation, we can apply \cite[Proposition 5.5.4.15]{HTT} to deduce that $\Psh_\sR(\cC)$ is not only cocomplete but also \textit{presentable}. We can now apply this to the case where $\sR$ is the collection of all colimiting cones in $\cC$.

\begin{defn}\label[defn]{Pres_Env}
    For every $\infty$-category $\cC$, we define 
    \[
        \widehat{j}\colon \cC \intoo \Psh_{\widehat{\sR}}(\cC) =: \widehat{\cC}
    \]
    where $\widehat{\sR}$ is the collection of all small colimiting cones in $\cC$. We say it exhibits $\widehat{\cC}$ as the \textit{presentable envelope} of $\cC$.
\end{defn}

\begin{rem}
    \cite[Theorem 5.3.6.2]{HTT} is more general in two aspects. First, it treats the case $\Psh_\sR^\sK(\cC)$, which adds to $\cC$ only colimits whose shape belongs to a certain collection $\sK$. The case $\Psh_\sR(\cC)$ is where $\sK$ is the collection of \textit{all} small $\infty$-categories. Second, one does not have to assume the cones in $\sR$ are colimiting cones. However, without this assumption, the functor $j_\sR$ need not be fully faithful.  
\end{rem}

We conclude this section with a remark about the interaction of presentable envelopes with \textit{weighted} colimits.

\begin{prop}\label[prop]{Pres_Env_Weight_Colim}
    Let $\cC$ be an $\infty$-category. The embedding of $\cC$ into its presentable envelope preserves all the small weighted colimits that exist in $\cC$.
\end{prop}
\begin{proof}
    Given a small diagram $X\colon I \to \cC$ and a weight $W\colon I^\op \to \cC$, the colimit of $X$ weighted by $W$ is the colimit of the diagram
    \[
        \Tw(I) \oto{\ (t,s)\ } I^\op \times I \oto{W \otimes X} \cC.
    \]
    Since $\Tw(I)$ is also small, and $\widehat{j} \colon \cC \into \widehat{\cC}$ preserves all small colimits, it follows that it also preserves the weighted colimit of $X$ by $W$.
\end{proof}

%%%%%%%%%%%%%%%%%%%%%%%%%%%%%%%%%%%%%%%%%%%%%%%%%%%%%%%%%%%%%%%%%%%%%%%%%%%%%%%%%%%%%%%%%%
%%%%%%%%%%%%%%%%%%%%%%%%%%%%%%%%%%%%%%%%%%%%%%%%%%%%%%%%%%%%%%%%%%%%%%%%%%%%%%%%%%%%%%%%%%
\section{Main results}
%%%%%%%%%%%%%%%%%%%%%%%%%%%%%%%%%%%%%%%%%%%%%%%%%%%%%%%%%%%%%%%%%%%%%%%%%%%%%%%%%%%%%%%%%%
%%%%%%%%%%%%%%%%%%%%%%%%%%%%%%%%%%%%%%%%%%%%%%%%%%%%%%%%%%%%%%%%%%%%%%%%%%%%%%%%%%%%%%%%%%

Let $\cC$ be a presentable $\infty$-category, and let $f \colon J \to I$, $X \colon I \to \cC$ be functors of $\infty$-categories. The comparison map 
\[
    \alpha \colon 
    \colim_J (X \circ f) \too 
    \colim_I X
\]
can be expressed using the calculus of Kan extensions. We have an adjunction
\[
    f_! \colon \Fun(J,\cC) \adj \Fun(I,\cC) \noloc f^*
\]
where $f^*$ is the restriction along $f$, and its left adjoint $f_!$ is the left Kan extension along $f$. Writing $q \colon I \to \pt$, we can describe $\alpha$ as the counit of the adjunction $f_! \dashv f^*$ whiskered by $q$:
\[
    \colim_J (X\circ f) =  
    q_!f_!f^* X \oto{\ \alpha \ } 
    q_! X = 
    \colim _I X. 
\]

The left Kan extension of a particular diagram $X\colon J \to \cC$ along a functor $f\colon J\to I$ can be considered more generally as an adjoint of $f^*\colon \Fun(I,\cC) \to \Fun(J,\cC)$ at the object $X$, without assuming presentability or the existence of a global adjoint for $f^*$. For future use, we recall the pointwise formula for it. 

\begin{prop}[{\cite[Lemma 4.3.2.13]{HTT}}]\label[prop]{Lan_Ptw}
    Let $X\colon J \to \cC$ be a functor of $\infty$-categories. For every functor $f\colon J \to I$ and $a \in I$ we have,
    \[
        (f_!X)(a) \simeq \colim(J_{/a} \to J \oto{X} \cC).
    \]
    Namely, if the $J_{/a}$-shaped colimit on the right exists in $\cC$ for all $a$ in $I$, then $f_!X$ exists and is given pointwise by the above formula.
\end{prop}

By the universal property of $\Psh(I)$ as the free co-completion, the diagram $X \colon I \to \cC$ factors uniquely through the Yoneda embedding
\[
    I \intoo \Psh(I) \oto{\ \widehat{X}\ } \cC,
\]
where the functor $\widehat{X}$ is the unique colimit preserving extension (called the \textit{Yoneda extension}) of $X$. This can also be thought of as the left Kan extension of $X$ along the Yoneda embedding.
Observing that weights on $I$ are precisely objects of $\Psh(I)$, we have the following re-interpretation of weighted colimits:

\begin{prop}[{\cite[Proposition~4.9]{haugseng2022co}}]
    Let $\cC$ be a presentable $\infty$-category. For every diagram $X\colon I \to \cC$ and weight $W\colon I^\op \to \Spc$, we have
    \[
         \colim^W_I X \simeq \widehat{X}(W).
    \]  
\end{prop}

Before proceeding, it is illuminating to consider some examples of weighted colimits. 

\begin{example}\label[example]{Weight_Exa}
    Let $X\colon I \to \cC$ be a functor of $\infty$-categories.
    \begin{enumerate}
        \item For $W := \pt_I$, the constant $\pt$-values weight, we recover the usual colimit,
        \[
            \colim_I^W X = \colim_I X.
        \]
        \item\label[equation]{item-2} For $W := \Map(-,a)$, we get
        \[
            \colim_I^W X = X(a).
        \]
        Thus, the representable presheaf $\Map(-,a)$ behaves as the atomic measure concentrated at~$a$.
    \end{enumerate}
\end{example}

This leads to the following general duality result for weighted colimits (\Cref{Colim_Duality_Intro}):

\begin{thm}\label[thm]{Colim_Duality}
    Let $\cC$ be an $\infty$-category and $f\colon I\to J$ a functor. For every diagram $X\colon I\to \cC$ and weight $W\colon J^\op \to \Spc$, we have an isomorphism
    \[
        \colim_J^W (f^* X) \simeq 
        \colim_I^{f^\op_! W} X \qin \cC.
    \]
\end{thm}
\begin{proof}
    By replacing $\cC$ with its presentable envelope $\widehat{\cC}$ from \cref{Pres_Env}, and taking into account \cref{Pres_Env_Weight_Colim}, we may assume without loss of generality that $\cC$ itself is presentable.
    Consider the commutative diagram arising from the universal property of the presheaf $\infty$-category as a free co-completion:
    % https://q.uiver.app/#q=WzAsNixbMiwwLCJJIl0sWzQsMCwiXFxjQyJdLFsyLDIsIlxcUHNoKEkpIl0sWzAsMCwiSiJdLFswLDIsIlxcUHNoKEopIl0sWzQsMiwiXFxjQyJdLFswLDEsIlgiXSxbMCwyLCIiLDIseyJzdHlsZSI6eyJ0YWlsIjp7Im5hbWUiOiJob29rIiwic2lkZSI6ImJvdHRvbSJ9fX1dLFszLDAsImYiXSxbMyw0LCIiLDIseyJzdHlsZSI6eyJ0YWlsIjp7Im5hbWUiOiJob29rIiwic2lkZSI6ImJvdHRvbSJ9fX1dLFs0LDIsIlxcd2lkZWhhdHtmfSJdLFszLDEsImZeKlgiLDAseyJjdXJ2ZSI6LTV9XSxbNCw1LCJcXHdpZGVoYXR7WFxcY2lyYyBmfSIsMix7ImN1cnZlIjo1fV0sWzIsNSwiXFx3aWRlaGF0e1h9Il0sWzEsNSwiIiwwLHsibGV2ZWwiOjIsInN0eWxlIjp7ImhlYWQiOnsibmFtZSI6Im5vbmUifX19XV0=&macro_url=https%3A%2F%2Fraw.githubusercontent.com%2FLior-Yanovski%2Fquiver%2Fmaster%2FLatex%2520Macros%2520preamble
    \[\begin{tikzcd}
    	J && I && \cC \\
    	\\
    	{\Psh(J)} && {\Psh(I)} && \cC
    	\arrow["f", from=1-1, to=1-3]
    	\arrow["{f^*X}", curve={height=-30pt}, from=1-1, to=1-5]
    	\arrow[hook', from=1-1, to=3-1]
    	\arrow["X", from=1-3, to=1-5]
    	\arrow[hook', from=1-3, to=3-3]
    	\arrow[equals, from=1-5, to=3-5]
    	\arrow["{\widehat{f}}", from=3-1, to=3-3]
    	\arrow["{\widehat{X\circ f}}"', curve={height=30pt}, from=3-1, to=3-5]
    	\arrow["{\widehat{X}}", from=3-3, to=3-5]
    \end{tikzcd}\]
    Furthermore, we have $\widehat{f} = f^\op_!$ (see \cite[Corollary F]{haugseng2023two} or \cite{ramzi2023elementary}). Namely, $\widehat{f}$ is itself the left adjoint of the restriction functor
    \[
        (f^\op)^* \colon 
        \Fun(I^\op, \Spc) \too 
        \Fun(J^\op, \Spc).
    \]
    Chasing the diagram, we get
    \[
        \colim_J^W (f^* X) \simeq 
        (\widehat{X\circ f})(W) \simeq
        \widehat{X}(f_!^\op(W)) \simeq
        \colim_I^{f^\op_! W} X.
    \]
\end{proof}

\begin{cor}\label[cor]{Cof_Quant}
    Let $\cC$ be an $\infty$-category and $f\colon I\to J$ a functor of $\infty$-categories. For every diagram $X\colon I\to \cC$, we have an isomorphism
    \[
        \colim_J (f^* X) \simeq 
        \colim_I^{|J_{(-)/}|} X \qin \cC.
    \]
\end{cor}
\begin{proof}
    Applying \Cref{Colim_Duality} in the special case of $W = \pt_J$, the constant weight on the point, we get
    \[
        \colim_J (f^* X) \simeq 
        \colim_J^{\pt_J} (f^* X) \simeq 
        \colim_I^{f^{\op}_!(\pt_J)} X.
    \]
    It remains to analyze the weight $f^{\op}_!(\pt_J)$. Using the pointwise formula for the left Kan extension (see \Cref{Lan_Ptw}), we get that its value at $a$ is given by the colimit over $(J^\op)_{/a}$ of the constant diagram on $\pt \in \Spc$. Thus,
    \[
        f^{\op}_!(\pt_J)(a) \simeq 
        |(J^\op)_{/a}| \simeq
        |(J_{a/})^\op)| \simeq
        |J_{a/}|.
    \]
\end{proof}

\begin{example}\label[example]{Duality_Exa}
    Another special case of \Cref{Colim_Duality} is when $W\colon J^\op \to \Spc$ is $\Map(-,b)$ for some $b\in J$. One the one hand, we get (see \Cref{Weight_Exa}\labelcref{item-2})
    \[
        \colim_J^{\Map(-,b)} (f^*X) = (f^*X)(b) = X(f(b)).
    \]
    On the other, by \Cref{Colim_Duality}, we get
    \[
        \colim_J^{\Map(-,b)} (f^*X) = 
        \colim_I^{f^\op_!\Map(-,b)} X.
    \]
    And indeed, one can compute directly using the pointwise formula for the left Kan extension (see \Cref{Lan_Ptw}) that $\Map(-,b) \simeq b_! \pt$ for $b \colon \pt \to J^{\op}$ the map picking out $b$, and hence 
    \[
        f^\op_!\Map(-,b) = 
        f^\op_!b_! \pt = 
        f^\op(b)_! \pt =
        \Map(-,f(b)).
    \]
\end{example}

\Cref{Cof_Quant} leads immediately to the criterion for $\cC$-cofinality stated in \Cref{Local_Cof_Crit_Intro}.

\begin{thm}\label[thm]{Local_Cof_Crit}
    Let $\cC$ be an $\infty$-category. A functor $f \colon J \to I$ is $\cC$-cofinal if $|J_{a/}|$ is $\cC$-acyclic for all $a \in I$. The converse holds provided $\cC$ admits all constant $\Map(a,b)$-shaped colimits for $a,b\in I$.
\end{thm}
\begin{proof}
    The constant weight on a point $\pt_I \colon I^\op \to \Spc$ is terminal in the category $\Psh(I)$, so there is a unique map $|J_{(-)/}| \to \pt_I$. This map induces a natural transformation of functors
    % https://q.uiver.app/#q=WzAsMixbMCwwLCJJXlxcb3BcXHRpbWVzIEkiXSxbMywwLCJcXGNDIl0sWzAsMSwifEpfeygtKS99fFxcb3RpbWVzIFgiLDAseyJjdXJ2ZSI6LTN9XSxbMCwxLCJcXHB0IFxcb3RpbWVzIFgiLDIseyJjdXJ2ZSI6M31dLFsyLDMsIiIsMCx7InNob3J0ZW4iOnsic291cmNlIjoyMCwidGFyZ2V0IjoyMH19XV0=&macro_url=https%3A%2F%2Fraw.githubusercontent.com%2FLior-Yanovski%2Fquiver%2Fmaster%2FLatex%2520Macros%2520preamble
    \[\begin{tikzcd}
    	{I^\op\times I} &&& \cC
    	\arrow[""{name=0, anchor=center, inner sep=0}, "{|J_{(-)/}| \otimes X}", curve={height=-18pt}, from=1-1, to=1-4]
    	\arrow[""{name=1, anchor=center, inner sep=0}, "{\pt \otimes X}"', curve={height=18pt}, from=1-1, to=1-4]
    	\arrow[between={0.2}{0.8}, Rightarrow, from=0, to=1]
    \end{tikzcd}\]
    which at every point $(a,b)\in I^\op \times I$ is given by the canonical map 
    \[
        X(b) \otimes |J_{a/}| \too X(b) \qin \cC.
    \]
    Thus, if all the spaces $|J_{a/}|$ are $\cC$-acyclic, then the above natural transformation is an isomorphism and hence, combined with \Cref{Cof_Quant}, we get
    \[
        \colim_J (X \circ f) \simeq
        \colim_I^{|J_{(-)/}|} X \simeq
        \colim_I^{\pt_I} X \simeq 
        \colim_I X.
    \]
    Conversely, assume $f$ is $\cC$-cofinal. For an object $a\in I$, viewed as a functor $a\colon \pt \to I$, and an object $X \in \cC$, viewed as a diagram $X \colon \pt \to \cC$, consider the diagram $a_!X \colon I \to \cC$. The pointwise formula for the left Kan extension (see \Cref{Lan_Ptw}) gives an isomorphism\footnote{This is in a sense dual to \Cref{Duality_Exa}.}
    \[
        (a_! X)(b) \simeq \Map(a,b) \otimes X \qin \cC.
    \]
    On the one hand, by cofinality of $f$ and writing $q\colon I \to \pt$, we have
    \[
        \colim_J f^* (a_! X) \simeq
        \colim_I a_! X \simeq 
        q_!a_! X \simeq 
        X.
    \]
    On the other hand, by \cref{Cof_Quant}, we have
    \[
        \colim_J f^* a_! X \simeq 
        \colim_I^{|J_{(-)/}|} a_! X \simeq 
        \big(\widehat{a_! X}\big)(|J_{(-)/}|).
    \]
    Now, $\widehat{a_! X}$, as the composition of two left Kan extensions along
    \[
        \pt \oto{\ a\ } I \oto{\ j\ } \Psh(I),
    \]
    is given by left Kan extension of $X$ along the composite map $\pt \to \Psh(I)$ that picks the representable functor $j(a) = \Map(-,a)$. Thus, for any $W \in \Psh(I)$ we have
    \[
        \big(\widehat{a_! X}\big)(W) \simeq
        W(a) \otimes X.
    \]
    Putting everything together we get 
    \[
        |J_{a/}| \otimes X \iso  X \qin \cC.
    \]
    Namely, $|J_{a/}|$ is $\cC$-acyclic for all $a\in I$.
\end{proof}

%%%%%%%%%%%%%%%%%%%%%%%%%%%%%%%%%%%%%%%%%%%%%%%%%%%%%%%%%%%%%%%%%%%%%%%%%%%%%%%%%%%%%%%%%%
%%%%%%%%%%%%%%%%%%%%%%%%%%%%%%%%%%%%%%%%%%%%%%%%%%%%%%%%%%%%%%%%%%%%%%%%%%%%%%%%%%%%%%%%%%
\section{Reduced symmetric algebras}
%%%%%%%%%%%%%%%%%%%%%%%%%%%%%%%%%%%%%%%%%%%%%%%%%%%%%%%%%%%%%%%%%%%%%%%%%%%%%%%%%%%%%%%%%%
%%%%%%%%%%%%%%%%%%%%%%%%%%%%%%%%%%%%%%%%%%%%%%%%%%%%%%%%%%%%%%%%%%%%%%%%%%%%%%%%%%%%%%%%%%

In this final section, we describe an application of \Cref{Local_Cof_Crit_Intro}, due to Betts and Dan-Cohen. Let $\cC$ be a presentably symmetric monoidal $\infty$-category. The free $\EE_\infty$-algebra on an object $X \in \cC$ is given by the \textit{symmetric algebra} construction
\[
    \Sym(X) :=
    \bigsqcup_{n\in \NN} X^{\otimes n}_{h\Sigma_n}
\]
and constitutes an object-wise formula for the left adjoint to the forgetful functor $\Alg_{\EE_\infty}(\cC) \to \cC$. Similarly, a pointed object $\one \to X$ in $\cC$ can be viewed as an $\EE_0$-algebra structure on $X$, and the left adjoint to the forgetful functor $\Alg_{\EE_\infty}(\cC) \to \Alg_{\EE_0}(\cC)$ is given by the \textit{reduced symmetric algebra} functor
\[
    \cl{\Sym} \colon \Alg_{\EE_0}(\cC) \too \Alg_{\EE_\infty}(\cC). 
\]

\begin{prop}[\cite{Ishai_Paper}]\label[prop]{Ishai}
    Let $\cC$ be a presentably symmetric monoidal $\infty$-category such that $B\Sigma_n$ is $\cC$-acyclic for all $n$ (e.g., rational stable or ordinary category). For every pointed object $\one \to X$ in $\cC$, we have an isomorphism,
    \[
        \cl{\Sym}(X) \ \simeq\
        \colim \big(
        \one \to X \to X^{\otimes 2}_{h\Sigma_2} \to X^{\otimes 3}_{h\Sigma_3} \to \dots \big)
        \qin \cC.
    \]
\end{prop}

For the convenience of the reader we sketch the argument for how this is derived from \Cref{Local_Cof_Crit_Intro}.

\begin{proof}[Proof sketch.]
    By the general yoga of operadic Kan extensions, for every map $f\colon \sP \to \sQ$ of $\infty$-operads, the forgetful functor 
    \[
        f^*\colon \Alg_\sQ(\cC) \too \Alg_\sP(\cC)
    \]
    admits a left adjoint 
    \[
        f_!\colon \Alg_\sP(\cC) \too \Alg_\sQ(\cC).
    \]
    A general formula for the corresponding left adjoint is derived in \cite[\S 3.1.3]{HA}, and a more accessible account can be found in \cite[Example 9.2]{chu2023free}. In our special case, $\sP = \EE_0$ and $\sQ = \EE_\infty$, this formula unpacks as follows. 
    Let $\Fin^\inj$ be the category of finite sets and injective morphisms with the disjoint union symmetric monoidal structure.
    An $\EE_0$-algebra $\one \to X$ in $\cC$ determines a symmetric monoidal functor
    \[
        F_X \colon \Fin^\inj \too \cC,
    \]
    which is characterized by sending a singleton $\{\star\}$ to $X$ and the unique map $\es \to \{\star\}$ to the structure map $\one \to X$.\footnote{The category $\Fin^\inj$ arises as $\EE_{0,\mathrm{act}}^\otimes$ serving as the \textit{symmetric monoidal envelope} of $\EE_0$.} We then have that $\cl{\Sym}(X)$ is simply the colimit of $F_X$ in $\cC$. Namely,
    \[
        \cl{\Sym}(X) \simeq 
        \colim_{S \in \Fin^\inj} \ X^{\otimes |S|} 
        \qin \cC.
    \]
    This colimit can be computed in steps. Let $\Fin^\inj_{\le n} \sseq \Fin^\inj$ be the full subcategory on sets of cardinality less than or equal to $n$. 
    Since 
    \[
        \Fin^\inj \simeq \colim\big(
            \Fin_{\le 0}^\inj \too 
            \Fin_{\le 1}^\inj \too
            \Fin_{\le 2}^\inj \too
            \dots 
        \big),
    \]
    we get by \cite[Example 2.5]{horev2017conjugates},
    \[
        \colim_{S\in \Fin^\inj} \ X^{\otimes |S|} = 
        \colim_{n\in \NN}
        \big(\colim_{S\in \Fin^\inj_{\le n}} \ X^{\otimes |S|}\big).
    \]
    Now, consider the inclusion $f\colon B\Sigma_n \into \Fin^\inj_{\le n}$ of the full subcategory on an $n$-element set. 
    The lax fiber over any $S$ in $\Fin^\inj_{\le n}$ is easily seen to be equivalent to $B\Sigma_{n-|S|}$. As we assumed these spaces are $\cC$-acyclic, we get by \cref{Local_Cof_Crit} that  $f$ is $\cC$-cofinal. Namely,
    \[
        \colim_{S\in \Fin^\inj_{\le n}} \ X^{\otimes |S|} \simeq
        X^{\otimes n}_{h\Sigma_n}.
    \]    
    Putting everything together we get the formula
    \[
        \cl{\Sym}(X) \simeq \colim
        \big(\one \to X \to X^{\otimes 2}_{h\Sigma_2} \to X^{\otimes 3}_{h\Sigma_3} \to \dots \big) \qin \cC.
    \]
\end{proof}

\begin{rem}
    Chasing through the proof, each map $X^{\otimes n-1}_{h\Sigma_{n-1}} \to X^{\otimes n}_{h\Sigma_{n}}$ in the sequential diagram above is induced by the $\Sigma_{n-1}$-equivariant map 
    \[
        X^{\otimes n-1} = X^{\otimes n-1}\otimes \one \too 
        X^{\otimes n-1}\otimes X = X^{\otimes n},
    \]
    by passing to $\Sigma_{n-1}$-orbits and composing with the canonical projection:
    \[
        X^{\otimes n-1}_{h\Sigma_{n-1}} \too 
        X^{\otimes n}_{h\Sigma_{n-1}} \too
        X^{\otimes n}_{h\Sigma_{n}}.
    \]
\end{rem}

We observe that the conclusion of \cref{Ishai} does not hold in every $\infty$-category $\cC$. 

\begin{example}
    Let $\cC = \Spc$ and consider $X = \pt$ pointed by the identity. The reduced free $\EE_\infty$-algebra on it is simply $\pt$ again, whereas the sequential colimit appearing in \Cref{Ishai} gives the non-contractible space $B\Sigma_\infty$.    
\end{example}

Incidentally, the Dold-Thom theorem says that for a nice pointed connected \textit{topological space} $X$, we have
\[
    \cl{\Sym}(X) \simeq \colim
    \big(\one \to X \to X^{2}_{\Sigma_2} \to X^{3}_{\Sigma_3} \to \dots \big)
    \qin \mathrm{Top}.
\]
Namely, we need to replace the \textit{homotopy} $\Sigma_n$-orbits by the \textit{strict} point-set $\Sigma_n$-orbits. In fact, this comes from the homotopy equivalence
\[
    X^{n}_{\Sigma_n} \simeq 
    \colim_{S\in \Fin^\inj_{\le n}} X^{|S|}.
\]
expressing the strict colimit in the ordinary category of topological spaces on the left-hand side as a colimit in the $\infty$-category of spaces on the right-hand side. The comparison map $X^{n}_{h\Sigma_n} \to X^{n}_{\Sigma_n}$ coming from the inclusion $B\Sigma_n \to \Fin^\inj_{\le n}$ is precisely the canonical map from the homotopy orbits to the strict orbits. 
Thus, in a precise sense, the Dold-Thom theorem is about the \textit{failure} of \Cref{Ishai} for $\cC = \Spc$.

\bibliographystyle{alpha}
\phantomsection\addcontentsline{toc}{section}{\refname}
\bibliography{four}

\end{document}